\documentclass[12pt]{amsart}
\usepackage{amsfonts}
\usepackage{amsmath}
\usepackage{color}
\usepackage[mathscr]{eucal}
\usepackage{amscd, latexsym, graphicx, mathrsfs, enumerate}
\usepackage{amssymb}

\makeindex

\newtheorem{thm}{{{Theorem}}}[section]
\newtheorem{prop}[thm]{{Proposition}}
\newtheorem{lem}[thm]{{Lemma}}
\newtheorem{cor}[thm]{{Corollary}}

\newtheorem{remark}[thm]{Remark}
\numberwithin{equation}{section}

\def\Z{\mathbb{Z}}
\def\Q{\mathbb{Q}}

\def\C{\mathbb{C}}
\def\A{\mathbb{A}}

\def\GL{{\mathop{\mathrm{GL}}}}

\def\min{{\mathrm{min}}}

\def\O{{\mathcal O}}

\def\fp{{\mathfrak{p}}}

\def\ds{\displaystyle}

\def\lra{{\longrightarrow}}

\def\w{{\varpi}}
\def\F{{\mathbb{F}}}

\def\la{\lambda}

\numberwithin{equation}{section}

\title[A remark on conductor, depth,  and principal congruence subgroups]
{A remark on conductor, depth and principal congruence subgroups}
\author{Michitaka Miyauchi and Takuya Yamauchi}
\date{\today}

\keywords{depth, conductor, principal congruence subgroups}
\thanks{The first author is partially supported by JSPS KAKENHI Grant Number (B) No.19H01778. The second author is partially supported by 
JSPS KAKENHI Grant Number (B) No.19H01778.}
\subjclass[2010]{}

\address{Michitaka Miyauchi \\
Graduate School of Education, Okayama Univ.\\
3-1-1, Tsushima-Naka, Kita-ku, Okayama 700-8530
, JAPAN}
\email{miyauchi@okayama-u.ac.jp}

\address{Takuya Yamauchi \\
Mathematical Inst. Tohoku Univ.\\
 6-3,Aoba, Aramaki, Aoba-Ku, Sendai 980-8578, JAPAN}
\email{takuya.yamauchi.c3@tohoku.ac.jp}

\begin{document}
\maketitle

\begin{abstract}
In this paper we study a relation between conductor, depth, and the level of principal congruence subgroups for 
irreducible admissible representations of $\GL_n(F)$ for a non-archimedean local field 
$F$ of characteristic zero. 
As a global application, we estimate conductors of unitary irreducible automorphic cuspidal 
representations of $\GL_n(\A_\Q)$ in terms of principal congruence subgroups. 
We also give an explicit formula for the dimension of fixed vectors with respect to 
principal congruence subgroups for irreducible admissible representations of $\GL_2(F)$ 
as a local application.   
\end{abstract}

\tableofcontents

\section{Introduction}
Let $F$ be a non-archimedean local field of characteristic zero and $\O=\O_F$ the ring of integers of $F$. 
Let $\fp=\w\O$ be the maximal ideal of $\O$ where $\w$ is a fixed uniformizer of $F$. Put $\F=\O/\fp$. 
For an irreducible generic admissible representation $\pi$ of $\GL_n(F)$, there are several important invariants characterizing 
$\pi$. In this paper, we focus on the conductor $c(\pi)$ of $\pi$ defined in \cite{JPSS} 
(there is a gap in the proof there but later it is modified in \cite{Jacquet},\cite{Mat}) and 
the depth $\rho(\pi)$ defined in \cite{MP}. 
Both of invariants are defined by using families of suitable 
open compact subgroups of $\GL_n(F)$. 
Therefore, it is quite natural to ask any substantial proportion of the principal congruence subgroups among such families. 

In this paper, we study a condition when $\pi$ has a non-trivial fixed vector by a principal 
congruence subgroup in terms of the conductor $c(\pi)$ through study of the depth 
$\rho(\pi)$.  
Let us fix some notation to explain our results. 
Put $K=\GL_n(\O)$. 
For each positive integer $m$, let $K_1(m)$ be the subgroup of $K$ consisting of any element $g$ whose 
reduction modulo $\fp$ satisfies that the $n$-th row of $g$ mod $\fp^m$ is $(0,\ldots,0,1)$. 
We also define the principal congruence subgroup $K(m)$ of level $m$ to be the kernel of 
the mod $\fp^m$-reduction $K\lra \GL_n(\O/\fp^m)$.  
The conductor $c(\pi)$ of $\pi$ is defined as the minimal non-negative integer $m$ of which $\pi$ has a non-trivial $K_1(m)$-fixed vector. 
Then we will prove the following:
\begin{thm}\label{main1}
Let $\pi$ be an irreducible essentially square integrable representation of $\GL_n(F)$
and $m$ a non-negative integer.
The following two conditions are equivalent:
\begin{enumerate}
\item $\pi$ has a non-trivial $K(m)$-fixed vector; 
\item $c(\pi)\le mn$. 
\end{enumerate}
\end{thm} 
As is known, any irreducible generic representation $\pi$ can be written by 
\begin{equation}\label{langlands}
\pi={\rm Ind}^{{\rm GL}_n(F)}_{P_{\underline{n}}(F)}(\tau_1\otimes \cdots \otimes \tau_k)
\end{equation}
where $P_{\underline{n}}$ is the standard parabolic subgroup associated to the 
partition $\underline{n}=(n_1,\ldots,n_k)$ of $n$ such that its Levi subgroup is 
isomorphic to ${\rm GL}_{n_1}\times \cdots \times {\rm GL}_{n_k}$ (cf. Theorem 9.3 of \cite{PR}).  Further each 
$\tau_i\ (1\le i\le k)$ is an irreducible essentially square integrable representation 
of $\GL_{n_i}(F)$. 
Then we will show the following:
\begin{thm}\label{main2}
Let $\pi$ be as in $($\ref{langlands}$)$ and $m$ a non-negative integer. 
The following two conditions are equivalent:
\begin{enumerate}
\item $\pi$ has a non-trivial $K(m)$-fixed vector; 
\item $c(\tau_i)\le mn_i$ for each $1\le i\le k$. 
\end{enumerate}
\end{thm} 
It follows immediately from Theorem \ref{main1} and Theorem \ref{main2} that 
\begin{cor}\label{cor-of-main-thms}
Let $\pi$ be an irreducible generic representation of $\GL_n(F)$
and $m$ a non-negative integer. The following two conditions are equivalent:
\begin{enumerate}
\item $\pi$ has a non-trivial $K(m)$-fixed vector but $\pi^{K(m-1)}=0$; 
\item $m< c(\pi)\le mn$. 
\end{enumerate}
where $\pi^{K(m-1)}=0$ means no condition when m=0. 

Further, if $\pi$ is essentially square integrable, then the condition $m< c(\pi)\le mn$ is 
replaced with $(m-1)n< c(\pi)\le mn$. 
\end{cor}

A motivation in our study comes up to authors in an estimation of upper and lower bounds 
of the conductor of automorphic families in the setting of low lying zeros 
\cite{ST}, \cite{KWY} (see also Remark \ref{mot}). 

This paper is organized as follows. 
In Section 2, we give a proof of Theorem \ref{main1} which is a balk of this paper and 
the general case follows through the discussion in Section 3. 
Finally, we give some applications in both of global and local situation. 
The global one is about an estimation of the global conductor of 
a cuspidal representation of $\GL_n(\A_\Q)$. 
Finally, the dimension of $K(m)$-fixed vectors in 
the case when $n=2$ will be discussed. 

\textbf{Acknowledgments.} We would like to thank professors 
H-H. Kim and S. Wakatsuki for inspiring the authors to address this problem. 
We thank the referee for helping us to correct and improve our paper greatly. 

\section{Essentially square integrable representations}

Let $\pi$ be an irreducible admissible representation of $\mathrm{GL}_n(F)$.
The depth $\rho(\pi)$ of $\pi$ is defined in \cite{MP}.
Due to \cite{BL},
we may use the lattice-theoretic realization of unrefined minimal 
$K$-types in \cite{Bushnell}.
Let $\mathfrak{A}$ be a hereditary order in $M_n(F)$.
We denote by $\mathfrak{P}$ the Jacobson radical of $\mathfrak{A}$.
There exists a positive integer $e(\mathfrak{A})$
which satisfies $\mathfrak{P}^{e(\mathfrak{A})} = \varpi \mathfrak{A}$. 
By \cite{MP} Theorem 5.2,
the depth $\rho(\pi)$ of $\pi$ is equal to the least $i/e(\mathfrak{A})$
such  that 
$\pi$ has a non-zero $1+\mathfrak{P}^{i+1}$-fixed vector
for a non-negative integer $i$.

\begin{prop}\label{square_int2}
Let $\pi$ be an irreducible essentially square integrable representation 
of $\GL_n(F)$.
For any positive integer $m$,
the following two conditions are equivalent:
\begin{enumerate}
\item[(i)] $\rho(\pi) \leq m-1$;
\item[(ii)] $\pi$ has a non-zero $K(m)$-fixed vector. 
\end{enumerate}
\end{prop} 
\begin{proof}
(i) $\Longrightarrow$ (ii).
Let $\mathfrak{A}$ be a hereditary order in $M_n(F)$
so that $\pi$ has a non-zero $1+\mathfrak{P}^{i+1}$-fixed vector 
and $\rho(\pi) = i/e(\mathfrak{A})$.
Replacing by a conjugate if necessary,
we may assume that $M_n(\mathcal{O}) \supset \mathfrak{A}
\supset \mathfrak{P} \supset \varpi M_n(\mathcal{O})$.
Suppose that $\rho(\pi) = i/e(\mathfrak{A})\leq m-1$.
Then we have
\[
1+\mathfrak{P}^{i+1} \supset 1+\mathfrak{P}^{e(\mathfrak{A})(m-1)+1}
=1+ \varpi^{m-1}\mathfrak{P}
\supset 1+ \varpi^{m}M_n(\mathcal{O})
= K(m).
\]
So $\pi$ contains a non-zero $K(m)$-fixed vector.

(ii) $\Longrightarrow$ (i).
Suppose that $\pi$ has a non-zero $K(m)$-fixed vector.
Set $\mathfrak{A} = M_n(\mathcal{O})$.
Then $\mathfrak{A}$ is a hereditary order in $M_n(F)$ 
with $e(\mathfrak{A}) = 1$.
Since $K(m) = 1+\varpi^m M_n(\mathcal{O})
= 1+\mathfrak{P}^m$,
the depth $\rho(\pi)$ of $\pi$ is less than or equal to 
$(m-1)/e(\mathfrak{A}) = m-1$.
\end{proof}

\begin{prop}\label{square_int}
Let $\pi$ be an irreducible essentially square integrable representation 
of $\GL_n(F)$.
For any positive integer $m$,
the following two conditions are equivalent:
\begin{enumerate}
\item[(i)] $\rho(\pi) \leq m-1$;
\item[(ii)] $c(\pi)\le mn$. 
\end{enumerate}
\end{prop} 
\begin{proof}
By \cite{LR} Theorem 3.1,
we have
$\displaystyle \rho(\pi) =\max\left\{ \frac{c(\pi)-n}{n},0\right\}$.
Suppose that $\rho(\pi) = 0$.
Then we have $c(\pi) \leq n$.
In this case, both conditions (i) and (ii) hold for any positive integers
$m$.
If $\displaystyle \rho(\pi) = \frac{c(\pi)-n}{n}$,
then 
we see that
\[
\rho(\pi) = \frac{c(\pi)-n}{n} \leq m-1 \Longleftrightarrow 
c(\pi)-n \leq mn -n \Longleftrightarrow  c(\pi) \leq mn.
\]
This completes the proof.
\end{proof}

We shall prove Theorem~\ref{main1}.
\renewcommand{\proofname}{Proof of Theorem~\ref{main1}.}
\begin{proof}
If $m$ is a positive integer,
then 
the conditions (1) and (2) are equivalent
by Propositions~\ref{square_int2} and \ref{square_int}.
Suppose that $m=0$.
Then $c(\pi) \leq m n = 0$ if and only if $c(\pi) =0$.
We note that $K_1(0) = K(0) = \mathrm{GL}_n(\mathcal{O})$.
By the definition of $c(\pi)$, we see that
$c(\pi) =0$ if any only if $\pi$ has a non-zero $K(0)$-fixed vector.
This completes the proof.
\end{proof}

\renewcommand{\proofname}{Proof.}

\section{Parabolically induced representations}

Let $\underline{n}=(n_1,\ldots,n_k)$ 
be a partition of $n$.
We denote by $P_{\underline{n}}$ the standard parabolic subgroup of
$\mathrm{GL}_n(F)$ associated to $\underline{n}$.
Then $P_{\underline{n}}$ is the group consisting of all the 
elements $p$ in $\mathrm{GL}_n(F)$ which have the form
\[
p = \begin{pmatrix}
p_{11} & p_{12} & \cdots & p_{1k}\\
 & p_{22} & \cdots & p_{2k}\\
  &&\ddots & \vdots\\
  &&& p_{kk}
\end{pmatrix},
\]
where $p_{ij} \in M_{n_i, n_j}(F)$.
The Levi subgroup $M$ of $P_{\underline{n}}$
is the group of the block diagonal matrices in $P_{\underline{n}}$,
and isomorphic to 
$\mathrm{GL}_{n_1}(F)\times \cdots \times\mathrm{GL}_{n_k}(F)$.
For any non-negative integer $m$, 
we define an open compact subgroup $K(m)_{n_i}$ of 
$\mathrm{GL}_{n_i}(F)$ by
\[
K(0)_{n_i} = \mathrm{GL}_{n_i}(\mathcal{O}),\, 
K(m)_{n_i} = 1+\varpi^m M_{n_i}(\mathcal{O}),\, m \geq 1.
\]

For a representation $\pi$ of a group $G$
and a subgroup $H$ of $G$,
we denote by $\pi^H$ the space of $H$-fixed vectors in $\pi$.

\begin{lem}\label{lem:31}
Let $\pi  = \mathrm{Ind}_{P_{\underline{n}}(F)}^{\mathrm{GL}_n(F)}
(\tau_1 \otimes \cdots \otimes \tau_k)$
be the parabolically induced representation
where $\tau_i$ is a smooth representation of $\mathrm{GL}_{n_i}(F)$,
for $1 \leq i \leq k$. Then $\pi$ has a non-trivial $K(m)$-fixed vector
if and only if $\tau_i^{K(m)_{n_i}} \neq \{0\}$
for all $1 \leq i \leq k$.
It also holds that 
\[
\dim \pi^{K(m)} = |P_{\underline{n}}\backslash \mathrm{GL}_n(F)/K(m)|
\prod_{i=1}^k \dim \tau_i^{K(m)_{n_i}},\, m \geq 0.
\]
\end{lem}
\begin{proof}
Let $\Omega$ be a complete system of representatives for
$P_{\underline{n}}\backslash \mathrm{GL}_n(F)/K(m)$.
Set  $\tau = \tau_1 \otimes \cdots \otimes \tau_k$.
By \cite{BZ} Lemma 2.24,
there is an isomorphism of complex vector spaces
\[
\pi^{K(m)} \simeq \{f: \Omega \rightarrow \tau\, |\, f(g) \in \tau^{M\cap gK(m) g^{-1}}, g \in \Omega\}.
\]
By the Iwasawa decomposition 
$\mathrm{GL}_n(F) = P_{\underline{n}}K(0)$,
we may take $\Omega$ so that $\Omega \subset K(0)$.
Since $K(0)$ normalizes $K(m)$,
we have
\[
\pi^{K(m)} \simeq \{f: \Omega \rightarrow \tau\, |\, f(g) \in \tau^{M\cap K(m)}, g \in \Omega\}.
\]
So the claim follows because
\[
M\cap K(m) = K(m)_{n_1}\times \cdots \times K(m)_{n_k}.
\]
and
\[
\tau^{M\cap K(m)} = \tau_1^{K(m)_{n_1}} \otimes \cdots \otimes \tau_k^{K(m)_{n_k}}.
\]
This completes the proof.
\end{proof}

We shall prove Theorem~\ref{main2}.
\renewcommand{\proofname}{Proof of Theorem~\ref{main2}.}
\begin{proof}
By Lemma~\ref{lem:31},
$\pi$ has a non-trivial $K(m)$-fixed vector
if and only if $\tau_i^{K(m)_{n_i}} \neq \{0\}$
for all $1 \leq i \leq k$.
Since  $\tau_i$ is irreducible and essentially square integrable,
it follows from Theorem~\ref{main1} that
the latter condition is equivalent to $c(\tau_i) \leq m n_i$
for each $1 \leq i \leq k$.
\end{proof}
\renewcommand{\proofname}{Proof.}

\section{Applications}
In this section, we give some applications of the previous results. 
First we give both of lower and upper bounds of the conductor for each automorphic cuspidal 
representation $\Pi$ of ${\rm GL}_n(\A_\Q)$ when $\Pi$ has a fixed vector for 
some principal congruence subgroup. 
Next, we estimate the number of $K(m)$-fixed vectors for any irreducible 
admissible representation of ${\rm GL}_2(F)$. 
When the representation in question has minimal conductor among twists, 
the result is well known due to Casselman \cite{Ca0} and Tunnell \cite{Tunnell}. However, here we do not assume this minimality condition on the representations. 

We should mention that emeritus professor Hiroyuki Yoshida at Kyoto university 
wrote a Japanese book entitled as the theory of automorphic forms. 
In Section 5, p.255 of Chapter IX, he mentioned that he gave several lectures in Paris
 around 1982  
and the contents are based on his lecture notes prepared by himself. 
They are closely related to our contents below (Section \ref{n=2}) but not a copy of them. 
Further, he also assumed the representations have the minimal conductors as in the 
case of  Casselman-Tunnell.  
The authors in this article hope his book would be published in English in near future. 
We should also mention that our method relies on techniques from Casselman-Tunnell 
while Yoshida took a down-to-earth (but beautiful, educative) method.

\subsection{Lower and upper bounds of the conductor}
Let  us work on the ring of adeles $\A_\Q$ of $\Q$ only for simplicity. 
A similar result for any number field would be quite easy to formulate, but it is 
cumbersome and therefore omitted.  
Let $\Pi=\otimes'_{p}\Pi_p$ be a unitary irreducible automorphic cuspidal representation of 
${\rm GL}_n(\A_\Q)$. By Corollary in p.190 of \cite{Shalika}, at each prime $p$, 
the local component $\Pi_p$ is 
generic. For each positive integer $N$, put $K(N):={\rm Ker}(
{\rm GL}_n(\widehat{\Z})\stackrel{{\rm mod}\ N}{\lra}{\rm GL}_n(\widehat{\Z}))$. 
We denote by $c(\Pi)$ the global conductor of $\Pi$ which is defined by 
$c(\Pi)=\ds\prod_{p<\infty} p^{c(\Pi_p)}$. 
\begin{thm}\label{estcond}
Let $\Pi=\otimes'_p \Pi_p$ be
a unitary irreducible automorphic cuspidal representation of $\GL_n(\A_\Q)$. 
Suppose that there exists a positive integer $N$ such that 
$\Pi^{K(N)}\neq 0$ but $\Pi^{K(d)}=0$ for any proper divisor $d$ of $N$. 
Then, it holds that 
$$\max\Big\{\prod_{p|N}p,\ N\cdot \prod_{p|N}p^{-1} \Big\}\le c(\Pi)\le N^n.$$
\end{thm}
\begin{proof}It suffice to consider each prime $p$ dividing $N$. 
Put $e_p:={\rm ord}_p(N)$. By Theorem \ref{main2} we have $c(\Pi_p)\le e_p n$. 
The upper bound follows from this. For the lower bound in the claim, 
we write $\Pi_p={\rm Ind}^{{\rm GL}_n(\Q_p)}_{P_{(n_1,\ldots,n_k)}(\Q_p)}(\tau_1
\otimes\cdots\otimes\tau_k)$ 
as explained in right after Theorem \ref{main1}.  
By Theorem \ref{main2} and the assumption, there exists $1\le i \le k$ such that 
$c(\tau_i)>(e_p-1)n_i\ge e_p-1$. Hence $c(\Pi_p)\ge e_p-1$. 
Clearly, $c(\Pi_p)\ge 1$ 
if $p|N$ by assumption. The claim follows from this. 
\end{proof}

\begin{remark}\label{mot}
Let $H$ be a reductive group over $\Q$ with a homomoprhism $r:H\lra \GL_n$. 
In several settings, automorphic cuspidal representations on $H(\A_\Q)$ are 
transferred to automorphic representations on $\GL_n(\A_\Q)$. 
A basic and important example is the Arthur's endoscopic classification of 
automorphic representations for classical groups. 
To apply this kind of results to, for example, equidistribution theorems as in \cite{ST}, \cite{KWY} in the level aspects, we need to know an explicit form of the fundamental 
lemma for Hecke elements. In most cases, it is done for the characteristic 
functions for principal congruence subgroups. Therefore, it is quite natural to 
ask any useful relations between various known 
invariants, as conductors or depths, of automorphic representations and 
the existence of fixed vectors in some principal congruence 
subgroups. Our study suggests to consider new forms with respect to principal 
congruence subgroups introduced in Definition 5.1 in \cite{KWY}. 
This might be helpful to guarantee Hypothesis 11.4, p.129 of \cite{ST}.
\end{remark}

\subsection{The number of $K(m)$-fixed vectors when $n=2$}\label{n=2}
To ask if a representation has a non-zero 
$K(0)$-fixed vector is easy. Therefore, we suppose $m$ is positive. 
For any quasi-character $\lambda:F^\times\lra \C^\times$ and a non-negative integer 
$r$, let $\delta(c(\lambda)\le r)=1$ if the inequality $c(\lambda)\le r$ holds and 
$\delta(c(\lambda)\le r)=0$ otherwise. 

For non-supercuspidal representations of ${\rm GL}_2(F)$,
we have the following
\begin{prop}\label{ps}
Let $r$ be a positive integer. 
It holds that 
\begin{enumerate}
\item if $\pi=\pi(\chi_1,\chi_2)$ be a principal series representation 
with quasi-characters $\chi_i:F^\times\lra \C^\times$ for $i=1,2$, then 
${\rm dim}\hspace{0.5mm}\pi^{K(r)}=q^{r-1}(q+1)\delta(c(\chi_1)\le r)\delta(c(\chi_2)\le r)$;
\item if $\pi=\chi\otimes {\rm St}$ be the twisted Steinberg representation 
by a quasi-character $\chi:F^\times\lra \C^\times$, then 
${\rm dim}\hspace{0.5mm}\pi^{K(r)}=(q^{r}+q^{r-1}-1)\delta(c(\chi)\le r)$.
\end{enumerate}
\end{prop} 
\begin{proof}For the first claim by Lemma \ref{lem:31}, we have 
$${\rm dim}\hspace{0.5mm}\pi^{K(r)}=
 |B(\O/\varpi^r\O)|^{-1}|{\rm GL}_2(\O/\varpi^r\O)|\delta(c(\chi_1)\le r)
 \delta(c(\chi_2)\le r). $$ 
Then the claim follows from 
 $|B(\O/\varpi^r\O)|^{-1}|{\rm GL}_2(\O/\varpi^r\O)|=q^{r-1}(q+1)$ since $r$ is positive. 
 
 For the second claim, recall the exact sequence 
 $$0\lra \chi\otimes {\rm St}\lra \pi(\chi|\ast|^{\frac{1}{2}},\chi|\ast|^{-\frac{1}{2}})
 \lra \chi\otimes \det\lra 0.$$
 Taking $K(r)$-fixed parts is an exact functor by Proposition 2.1.7 of \cite{Ca}. 
 The claim follows from this by using Lemma \ref{lem:31} again. 
\end{proof}
\begin{remark}\label{general-ps}For any principal series representation $\pi$ of $\GL_n(F)$ 
which is not necessarily irreducible, by using Lemma \ref{lem:31}, 
we would have a similar result as in Proposition \ref{ps}-(1). We left this problem to the 
interesting readers. 
\end{remark}

We shall consider supercuspidal representations of ${\rm GL}_2(F)$.
For two quasi-character $\lambda_1,\lambda_2:F^\times\lra \C^\times$, we write $\la_1\sim \la_2$ if $\la_1\la^{-1}_2$ is unramified. 
For each non-negative integer $r$, let $X_{\le r}$ (resp. $X_r$) be the set of all quasi-characters 
with conductors less than or equal to $r$ (resp. equal to $r$).  
Choose a complete system $L_{\le r}$ (resp. $L_r$) of representatives of $X_{\le r}/\sim$ 
(resp. $X_{r}/\sim$). 
Notice that $X_{\le r}/\sim$ is isomorphic to $\O^\times_F/(1+\varpi^r\O_F)$ 
via the Pontryagin dual and it follows from this that 
$$|L_0|=1,\ |L_1|=q-2,\ |L_i|=(q-1)^2q^{i-2}\ (i\ge 2).$$ 
For each character $\la:F^\times\lra \C^\times$, we define the function $\xi^{(m)}_{\la}$ 
on $F^\times$ by  
$$\xi^{(m)}_{\la}(x)=\left\{\begin{array}{cl}
\la(x) & \text{if ${\rm ord}_{\varpi}(x)=-m$}\\
0 & \text{otherwise}
\end{array}\right..
$$
Notice that for an unramified character $\la':F^\times\lra\C^\times$, we see easily that 
$\xi^{(m)}_{\la\la'}=\la'(\varpi)^{-m}\xi^{(m)}_{\la}$. 
\begin{lem}\label{generators}
Let $\psi$ be a non-trivial additive character of $F$ with the conductor $c(\psi)$ 
$($hence, the largest integer which satisfies $\psi(\varpi^{-c(\psi)})=1)$. 
Let $\pi$ be an irreducible supercuspidal representation of $\GL_2(F)$ and denote by 
$\mathcal{K}_\psi(\pi)$ the Krillov model of $\pi$ with respect to $\psi$.
Suppose $r\ge -c(\psi)$. Then $\mathcal{K}_\psi(\pi)^{K(r)}$ has a basis $B_r$ 
consisting of all $\xi^{(m)}_\la,\ \la\in L_{\le r}$ satisfying $c(\pi\otimes \la)+c(\psi)-r\le 
m\le c(\psi)+r$.  
\end{lem}
\begin{proof}This is well-known for experts but we give a sketch of a proof 
(see also p.102 of \cite{S}). Since $F^\times$ is abelian and 
$F^\times=\ds\bigcup_{n\in \Z}\varpi^n\O^\times$, 
by using the character expansion, each  element $v$ of $\mathcal{K}_\psi(\pi)^{K(r)}$ 
can be written as a finite sum $v=\ds\ds\sum_{\la}\sum_{m\in\Z}a_{m,\la}\xi^{(m)}_\la$ where 
$a_{m,\la}=\ds\int_{\varpi^{-m}\O^\times}\la(x)
\pi(\begin{pmatrix}
x & 0 \\
0& 1 
\end{pmatrix}
)vdx^\times$ and $\la$ runs over a finite subset of 
$\ds\bigcup_{r\ge 0} L_r$. Since $v$ is $K(r)$-fixed, $a_{m,\la}=0$ unless $c(\la)\le r$. 
Since $v$ is fixed by  $\begin{pmatrix}
1 & \varpi^r \\
0& 1 
\end{pmatrix}$, it follows from $\begin{pmatrix}
x & 0 \\
0& 1 
\end{pmatrix}\begin{pmatrix}
1 & \varpi^r \\
0& 1 
\end{pmatrix}\begin{pmatrix}
x & 0 \\
0& 1 
\end{pmatrix}^{-1}=\begin{pmatrix}
1 & x\varpi^r \\
0& 1 
\end{pmatrix}$ that $$a_{m,\la}=\ds\int_{\varpi^{-m}\O^\times}\la(x)\psi(x\varpi^r)
\pi(\begin{pmatrix}
x & 0 \\
0& 1 
\end{pmatrix}
)vdx^\times=\psi(\varpi^{r-m})a_{m,\la}.$$ Therefore, $a_{m,\la}=0$ unless $r-m\ge -c(\psi)$. 
Let $w=\begin{pmatrix}
0 & 1 \\
-1& 0 
\end{pmatrix}$ be the Weyl element of $\GL_2(F)$. By Lemma 2.1 of \cite{S} 
which is the formula (9) of H. Yoshida \cite{Yoshida}, we have 
$\pi(w)\xi^{(m)}_\la=\varepsilon(\pi\otimes\la^{-1}) \xi^{(m_w)}_{\omega_\pi\la^{-1}}$ 
where $\varepsilon(\pi\otimes\la^{-1})$ is the epsilon factor and $m_w=c(\pi\otimes\la^{-1})+2c(\psi)-m$. Since $w^{-1}
\begin{pmatrix}
1 & \varpi^r \\
0& 1 
\end{pmatrix}w=\begin{pmatrix}
1 &0 \\
- \varpi^r& 1 
\end{pmatrix}$, the similar argument shows that $a_{m,\la}=0$ unless 
$r-m_w\ge -c(\psi)$. Summing up we have 
\begin{equation}\label{tc}
c(\la)\le r,\ c(\pi\otimes\la^{-1})+c(\psi)-r\le m\le r+c(\psi).
\end{equation}
This yields $\mathcal{K}_\psi(\pi)^{K(r)}\subset B_r$. 

Conversely, the above argument also shows $\xi^{(m)}_\la$ is fixed by $K(r)$ 
if and only if in addition to above two conditions (\ref{tc}), $r\ge -c(\psi)$. 
This completes a proof. 
\end{proof}

\begin{lem}\label{sai}Let $\pi$ be an 
 irreducible supercuspidal
 representation of ${\rm GL}_2(F)$ and $r$ a positive integer. 
If $c(\pi)>2r$, then $\pi^{K(r)}=0$. Contrary, if  $\pi^{K(r)}\neq 0$ 
$($hence $c(\pi)\le 2r)$, then 
$${\rm dim}\hspace{0.5mm}\pi^{K(r)}=
\sum_{i=0}^r\sum_{\la\in L_i}(2r-c(\pi\otimes\la^{-1})+1)=
\sum_{i=0}^r\sum_{\la\in L_i}(2r-c(\pi\otimes\la)+1).$$   
\end{lem} 
\begin{proof}
The former claim follows from Theorem \ref{main1} for $n=2$ and $m=r$. 
The latter formula is given by the counting all elements in the set $B_r$ in Lemma 
\ref{generators} for an additive character $\psi$ of $F$ with $c(\psi)=0$.
\end{proof}

Let $\pi$ is an irreducible supercuspidal representation of ${\rm GL}_2(F)$. 
By Theorem~\ref{main1},
we have $\dim \pi^{K(m)} = 0$ unless $c(\pi) \leq 2m$.
It follows from \cite{LR} Proposition 2.2
that
$c(\pi) = 2\rho(\pi)+2 \geq 2$.
%It is easy to see that $c(\pi)\ge 2$ (cf. Proposition in Section 14.3 with the comments in 
%the third paragraph in p. 104 of \cite{BH}).   
We say that
$\pi$ is minimal if 
\[
c(\pi) \leq c(\pi \otimes \chi) 
\]
for any quasi-character $\chi$ of $F^\times$.
Suppose that $\pi$ is minimal.
Then it follows from \cite{Tunnell} Proposition 3.5 that
\begin{align}\label{eq:Tunnell}
c(\pi \otimes \chi) 
= \left\{
\begin{array}{cl}
c(\pi) & \mbox{if}\  2c(\chi)\leq c(\pi),\\
2c(\chi) & \mbox{if}\  c(\pi) < 2c(\chi).
\end{array}
\right.
\end{align}

\begin{prop}\label{prop:43}
Let $\pi$ be an irreducible minimal supercuspidal representation of 
${\rm GL}_2(F)$. 
Set 
$r  = \left\lfloor \frac{c(\pi)}{2}\right\rfloor$.
Then, for any integer $m$ which satisfies $c(\pi) \leq 2m$,
we have
$${\rm dim}\hspace{0.5mm}\pi^{K(m)}=
 (2m-c(\pi)+1) (q-1)q^{r-1}
+ \sum_{i=r+1}^{m}(2(m-i)+1)(q-1)^2q^{i-2}.
$$
\end{prop}
\begin{proof}
Note that $2r \leq c(\pi) < 2r+2$.
By Lemma~\ref{sai} and (\ref{eq:Tunnell}),
we get
\begin{align*}
{\rm dim}\hspace{0.5mm}\pi^{K(m)}& =
\sum_{i=0}^m\sum_{\la\in L_i}(2m-c(\pi\otimes\la)+1)\\
& = \sum_{i=0}^{r} \sum_{\la\in L_i} (2m-c(\pi)+1)
+ \sum_{i=r+1}^{m} \sum_{\la\in L_i} (2m-2c(\lambda)+1)\\
& =  (2m-c(\pi)+1)(q-1)q^{r-1}
+ \sum_{i=r+1}^{m}(2(m-i)+1)(q-1)^2q^{i-2},
\end{align*}
as required.
\end{proof}

\begin{cor}
Let $\pi$ be an irreducible minimal supercuspidal representation of 
${\rm GL}_2(F)$. 
Set 
$r = \left\lfloor \frac{c(\pi)}{2}\right\rfloor$.
\begin{enumerate}
\item[(i)] If $c(\pi)$ is even, then 
$\dim \pi^{K(r)} 
=  (q-1)q^{r-1}$.
\item[(ii)] If $c(\pi)$ is odd, then 
$\dim \pi^{K(r+1)} 
=   (q+1)(q-1)q^{r-1}$.
\end{enumerate}
\end{cor}

In general,
for an irreducible supercuspidal representation $\pi$ of ${\rm GL}_2(F)$,
there exist an irreducible minimal supercuspidal
representation $\tau$ of ${\rm GL}_2(F)$ and 
a quasi-character $\chi$ of $F^\times$
such that 
$\pi = \tau \otimes \chi$. 
In fact, being supercuspidal is preserved under twisting by characters by the formula (9.5.2) of \cite{BH}. The existence of the minimal conductor among twists of $\pi$ guarantees. 
Thus, we can take a quasi character $\chi$
so that $c(\pi \otimes \chi^{-1}) = \min\{c(\pi\otimes \chi')\, |\, 
\chi': \mbox{a quas-character of }\ F^\times\}$.
Then $\tau = \pi \otimes \chi^{-1}$ is an irreducible minimal supercuspidal
representation which satisfies $\pi = \tau \otimes \chi$.

\begin{prop}\label{prop:45}
Let $\tau$ be an irreducible minimal supercuspidal representation
of ${\rm GL}_2(F)$
and $\chi$ a quasi-character of $F^\times$.
Set $\pi = \tau \otimes \chi$.
For any integer $m$ such that $c(\pi) \leq 2m$,
we have
\[
\dim \pi^{K(m)} = \dim \tau^{K(m)}.
\]
\end{prop}
\begin{proof}
By (\ref{eq:Tunnell}),
we see that $2c(\chi) \leq c(\pi)$.
So we have $c(\chi) \leq m$.
Since the character $\chi \circ \det$ of  ${\rm GL}_2(F)$
is trivial on $K(m)$,
the functor $\tau \mapsto \tau \otimes \chi$
preserves the $K(m)$-fixed vectors.
This implies the assertion.
\end{proof}

By Propositions~\ref{prop:43} and \ref{prop:45},
we obtain the following
\begin{cor}
Let $\pi$ be an irreducible supercuspidal representation
of ${\rm GL}_2(F)$.
Set
\[
s = \min\{c(\pi\otimes \chi)\, |\, \chi:
\mbox{a quasi-character of $F^\times$}\}\ and\ 
r = \left\lfloor \frac{s}{2}\right\rfloor.
\]
Then for any integer $m$ such that $c(\pi) \leq 2m$,
we have
\[
\dim \pi^{K(m)} = 
(2m-s+1) (q-1)q^{r-1}
+ \sum_{i=r+1}^{m}(2(m-i)+1)(q-1)^2q^{i-2}.
\]
\end{cor}

\end{document}